\documentclass[preprint,12pt,3p]{elsarticle}
\usepackage{amsmath,amsthm,amsfonts,amssymb}
\usepackage{hyperref}
\usepackage{yhmath}
\usepackage{MnSymbol}  
\usepackage[normalem]{ulem}
\usepackage{sectsty}

\usepackage{colortbl}
\usepackage[dvipsnames]{xcolor}
\newtheorem{theorem}{Theorem}[section]
\newtheorem{lemma}[theorem]{Lemma}

\newtheorem{cor}{Corollary}[section]

\numberwithin{equation}{section}

\definecolor{darkslategray}{rgb}{0.18, 0.31, 0.31}
\definecolor{warmblack}{rgb}{0.0, 0.26, 0.26}
\definecolor{astral}{RGB}{46,116,181}
\subsectionfont{\color{astral}}
\sectionfont{\color{astral}}
\journal{....................... }

\begin{document}
	\begin{frontmatter}
		\title{ \textcolor{warmblack}{\bf On $\mathbb{A}$-numerical radius inequalities for $2\times2$ operator matrices-II }}

	\author[label1]{Satyajit Sahoo}\ead{satyajitsahoo2010@gmail.com}

		\address[label1]{P.G. Department of Mathematics, Utkal University, Vanivihar, Bhubaneswar-751004, India}

		\begin{abstract}
			\textcolor{warmblack}{
 				 The main goal of this article is to establish  several new upper and lower bounds for the $\mathbb{A}$-numerical radius of $2\times 2$ operator matrices, where $\mathbb{A}$ be the $2\times 2$ diagonal operator matrix whose diagonal entries are positive bounded operator $A$.}
		\end{abstract}
		
 	\end{frontmatter}
	
		\section{Introduction}\label{intro}
	Let $\mathcal{H}$ be a complex Hilbert space with inner product $\langle \cdot,\cdot\rangle$ and ${\mathcal{L}}(\mathcal{H})$ be the $C^*$-algebra of all bounded linear operators on $\mathcal{H}$.
	The {\it numerical range} of $T\in
	\mathcal{B(H)}$ is defined as
	$$W(T)=\{\langle Tx, x \rangle: x\in \mathcal{H}, \|x\|=1
	\}.$$ The {\it numerical radius} of $T$, denoted by $w(T)$,  is
	defined as $ w(T)=\displaystyle\sup\{|z|: z\in W(T) \}.$
	It is well-known that
	$w(\cdot)$ defines a norm on $\mathcal{H}$, and is equivalent to the
	usual operator norm $\|T\|=\displaystyle \sup  \{ \|Tx \|: x\in \mathcal{H}, \|x\|=1 \}.$ In fact, for
	every $T \in \mathcal{B(H)}$, 
	\begin{align}\label{p3100}
	\frac{1}{2}\|T\|\leq w(T)\leq \|T\|.
	\end{align}
	An interested reader is referred to the recent articles \cite{MBKS,TY,SND,SND1,SND2} for different generalizations, refinements and applications of numerical radius inequalities.

	Let $\|\cdot\|$ be the norm induced from $\langle \cdot,\cdot\rangle.$
	An operator $A\in{\mathcal{L}}(\mathcal{H})$ is called {\it selfadjoint} if $A=A^{*}$, where $A^{*}$ denotes the adjoint of $A$.  A selfadjoint operator $A\in{\mathcal{L}}(\mathcal{H})$ is called {\it positive} if $\langle Ax, x\rangle \geq 0$ for all $x \in \mathcal{H}$, and is called {\it strictly positive} if  $\langle Ax, x\rangle > 0$ for all non-zero $x\in \mathcal{H}$. We denote a positive (strictly positive) operator $A$ by $A \geq 0$ ($A > 0$).  We denote $\mathcal{R}(A)$ as the range space of $A$ and $\overline{\mathcal{R}(A)}$ as the norm
	closure of $\mathcal{R}(A)$ in $\mathcal{H}$. Let $\mathbb{A}$ be a $2\times 2$ diagonal operator matrix whose diagonal entries are positive operator $A$. Then $\mathbb{A} \in{\mathcal{L}}(\mathcal{H}\bigoplus \mathcal{H})$ and $\mathbb{A}\geq 0$. If $A\geq 0$, then it induces a positive semidefinite sesquilinear form, $\langle \cdot,\cdot\rangle_A: \mathcal{H}\times\mathcal{H}\rightarrow\mathbb{C}$ defined by $\langle x, y \rangle_A=\langle Ax,y\rangle,$ $x,y\in\mathcal{H}.$ 
	Let $\|\cdot\|_A$ denote the seminorm on $\mathcal{H}$ induced by $\langle \cdot, \cdot \rangle_A,$ i.e., $\|x\|_A=\sqrt{\langle x, x \rangle_A}$ for all $x \in \mathcal{H}.$ Then $\|x\|_A$ is a norm
	if and only if $A>0$. Also, $(\mathcal{H}, \|\cdot\|_A)$ is complete if and only if  $\mathcal{R}(A)$ is closed in $\mathcal{H}.$ Here onward, we fix $A$ and $\mathbb{A}$ for positive operators on $\mathcal{H}$ and $\mathcal{H}\bigoplus \mathcal{H}$, respectively. We also reserve the notation $I$ and $O$ for the identity operator  and the null operator on $\mathcal{H}$ in this paper.

	$\|T\|_A$ denotes  the $A$-operator seminorm of $T \in{\mathcal{L}}\mathcal{(H)}$. This 
	is defined as follows:
	$$\|T\|_A=\sup_{x\in \overline{\mathcal{R}(A)},~x\neq 0}\frac{\|Tx\|_A}{\|x\|_A}=\inf\left\{c>0: \|Tx\|_A\leq c\|x\|_A, 0\neq x\in \overline{\mathcal{R}(A)}\right\}<\infty.$$
	Let $${\mathcal{L}}^A\mathcal{(H)}=\{T\in \mathcal{B(H)}:\|T\|_A<\infty\}.$$ Then ${\mathcal{L}}^A\mathcal{(H)}$ is not a subalgebra of $\mathcal{B(H)}$, and $\|T\|_A=0$ if and only if $ATA=O.$ For $T\in{\mathcal{L}}^A\mathcal{(H)},$ we have 
	$$\|T\|_A=\sup \{|\langle Tx,y\rangle_A|: x,y\in \overline{\mathcal{R}(A),} ~\|x\|_A=\|y\|_A=1\}.$$
	If $AT\geq 0$, then the operator $T$ is called {\it $A$-positive}.  Note that if $T$ is $A$-positive, then 
	$$\|T\|_A=\sup \{\langle Tx,x\rangle_A: x\in \mathcal{H},  \|x\|_A=1\}.$$
	An operator $X\in \mathcal{B(H)}$ is called an {\it $A$-adjoint operator} of $T\in \mathcal{B(H)}$ if  $\langle Tx,y\rangle_A=\langle x, Xy\rangle_A$  for every $x,y\in \mathcal{H},$ i.e., $AX=T^*A.$
	By Douglas Theorem \cite{Doug}, the existence of an $A$-adjoint operator is not guaranteed.  An operator $T\in \mathcal{B(H)}$ may admit none, one or many $A$-adjoints.  $A$-adjoint  of an operator $T\in {\mathcal{L}}\mathcal{(H)}$ exists if and only if $\mathcal{R}(T^*A)\subseteq \mathcal{R}(A)$. Let us now denote  $${\mathcal{L}}_A\mathcal{(H)}=\{T\in \mathcal{B(H)}:\mathcal{R}(T^*A)\subseteq \mathcal{R}(A)\}.$$ Note that ${\mathcal{L}}_A\mathcal{(H)}$ is a subalgebra of $\mathcal{B(H)}$ which is neither closed nor dense in $\mathcal{B(H)}.$ Moreover, the following inclusions $${\mathcal{L}}_A\mathcal{(H)}\subseteq {\mathcal{L}}^A\mathcal{(H)}\subseteq{\mathcal{L}}\mathcal{(H)}$$ hold with equality if $A$ is injective and has a closed range.

	The {\it Moore-Penrose inverse} of $A\in \mathcal{B(H)}$ \cite{Nashed} is the
	operator $X:R(A)\bigoplus R(A)^{\perp} \longrightarrow \mathcal{H}$ which satisfies the following four equations:
	\begin{center}
		(1) $AXA = A$,~ (2) $XAX = X$,~ (3) $X A= P_{N(A)^{\perp}}$,~~(4) $A X=P_{\overline{\mathcal{R}(A)}}|_{R(A)\bigoplus R(A)^{\perp}}$.
	\end{center}
	Here $N(A)$  and  $P_L$ denote the null space of $A$ and the orthogonal projection onto $L$, respectively.
	The Moore-Penrose inverse
	is unique, and  is denoted by $A^\dagger.$ In general, $A^{\dagger}\notin \mathcal{B(H)}$. It is bounded if and only if  $\mathcal{R}(A)$  is closed. 
	If $A\in \mathcal{B(H)}$ is invertible, then $A^\dagger= A^{-1}.$ 
	If $T\in {\mathcal{L}}_A\mathcal{(H)},$ the reduced solution of the equation $AX=T^*A$ is a distinguished $A$-adjoint operator of $T,$ which is denoted by $T^{\#_A}$ (see \cite{ARIS2,Mos}). Note that $T^{\#_A}=A^\dagger T^* A$. If $T\in {\mathcal{L}}_A(\mathcal{H}),$ then $AT^{\#_A}=T^*A$, $\mathcal{R}(T^{\#_A})\subseteq \overline{\mathcal{R}(A)}$ and $\mathcal{N}(T^{\#_A})=\mathcal{N}(T^*A)$ (see \cite{Doug}).  An operator $T\in \mathcal{B(H)}$ is said to be {\it $A$-selfadjoint} if $AT$ is selfadjoint, i.e., $AT=T^*A.$ Observe that if $T$ is $A$-selfadjoint, then $T\in {\mathcal{L}}_A(\mathcal{H}).$  However,  in general, $T\neq T^{\#_A}.$ But, $T=T^{\#_A}$  if and only if $T$ is $A$-selfadjoint and $\mathcal{R}(T)\subseteq \overline{\mathcal{R}(A)}.$ If $T\in {\mathcal{L}}_A(\mathcal{H}),$ then $T^{\#_A}\in {\mathcal{L}}_A(\mathcal{H}),$  $(T^{\#_A})^{\#_A}=P_{\overline{\mathcal{R}(A)}}TP_{\overline{\mathcal{R}(A)}},$  and $\left((T^{\#_A})^{\#_A}\right)^{\#_A}=T^{\#_A}.$ Also, $T^{\#_A}T$ and $TT^{\#_A}$ are  $A$-positive operators, and
	\begin{align}\label{ineq0}
	\|T^{\#_A}T\|_A=\|TT^{\#_A}\|_A=\|T\|_A^2=\|T^{\#_A}\|_A^2=w_A(TT^{\#_A})=w_A(T^{\#_A}T).
	\end{align}
An operator $T$ is called $A$-bounded if there exists $\alpha>0$ such that $ \|Tx\|_{A} \leq \alpha \|x\|_{A},\;\forall\,x\in \mathcal{H}.$  By applying Douglas theorem, one can easily see that the subspace of all operators admitting $A^{1/2}$-adjoints, denoted by ${\mathcal{L}}_{A^{1/2}}(\mathcal{H})$, is equal the collection of all $A$-bounded operators, i.e.,
		$${\mathcal{L}}_{A^{1/2}}(\mathcal{H})=\left\{T \in {\mathcal{L}}(\mathcal{H})\,;\;\exists \,\alpha > 0\,;\;\|Tx\|_{A} \leq \alpha \|x\|_{A},\;\forall\,x\in \mathcal{H}  \right\}.$$
		Notice that ${\mathcal{L}}_{A}(\mathcal{H})$ and ${\mathcal{L}}_{A^{1/2}}(\mathcal{H})$ are two subalgebras of ${\mathcal{L}}(\mathcal{H})$ which are, in general, neither closed nor dense in ${\mathcal{L}}(\mathcal{H})$. Moreover, we have ${\mathcal{L}}_{A}(\mathcal{H})\subset {\mathcal{L}}_{A^{1/2}}(\mathcal{H})$ (see \cite{ARIS2,acg3}).
	
	An operator $U\in {\mathcal{L}}_A(\mathcal{H})$ is said to be {\it $A$-unitary} if $\|Ux\|_A=\|U^{\#_A}x\|_A=\|x\|_A$ for all $x\in \mathcal{H}.$ For $T,S\in {\mathcal{L}}_A(\mathcal{H}),$ we have $(TS)^{\#_A}=S^{\#_A}T^{\#_A},$  $(T+S)^{\#_A}=T^{\#_A}+S^{\#_A},$ $\|TS\|_A\leq \|T\|_A\|S\|_A$ and $\|Tx\|_A\leq \|T\|_A\|x\|_A$ for all $x\in \mathcal{H}.$
	In 2012, Saddi \cite{Saddi} introduced  {\it $A$-numerical radius} of $T$ for $T\in \mathcal{B(H)}$, which is denoted as $w_A(T)$, and is defined as follows:
	\begin{equation}\label{eqn1a}
	w_A(T)=\sup\{|\langle Tx,x\rangle_A|:x\in \mathcal{H}, \|x\|_A=1\}.
	\end{equation}
	From \eqref{eqn1a}, it follows that $$w_A(T)=w_A(T^{\#_A}) \mbox{  for   any  }   T\in{\mathcal{L}}_A(\mathcal{H}).$$
	 A fundamental inequality for the $A$-numerical radius is the power inequality (see \cite{MOS}) which says that for $T\in \mathcal{B(H)}$, 
	\begin{align}\label{power}
	w_A(T^n)\leq w_A^n(T), ~~n\in\mathbb{N}.
	\end{align}   
	
	Notice that the $A$-numerical radius of semi-Hilbertian space operators satisfies the weak $A$-unitary invariance property which asserts that
	\begin{equation}\label{weak}
w_{A}(U^{\#_A}TU)=w_{A}(T),
	\end{equation}
	for every $T\in {\mathcal{L}}_{A}(\mathcal{H})$ and every $A$-unitary operator $U\in {\mathcal{L}}_{A}(\mathcal{H})$ (see \cite[Lemma 3.8]{bfeki}).
	
	An interested reader may refer \cite{ARIS,ARIS2} for further properties of operators on Semi-Hilbertian space.\\

	Let $$\Re_A(T):=\frac{T+T^{\#_A}}{2}\;\;\text{ and }\;\;\Im_A(T):=\frac{T-T^{\#_A}}{2i},$$ for any arbitrary operator $T\in {\mathcal B}_A({\mathcal H})$.
	Recently, in 2019 Zamani \cite[Theorem 2.5]{Zam} showed that if $T\in{\mathcal{L}}_{A}(\mathcal{H})$, then
	\begin{align}\label{zm}
	w_A(T) = \displaystyle{\sup_{\theta \in \mathbb{R}}}{\left\|\Re_A(e^{i\theta}T)\right\|}_A=\displaystyle{\sup_{\theta \in \mathbb{R}}}{\left\|\Im_A(e^{i\theta}T)\right\|}_A.
	\end{align}
	In 2019, Zamani \cite{Zam} showed that if $T\in {\mathcal{L}}_A\mathcal{(H)}$, then 
	\begin{align}\label{ineq00}
	w_A(T)=\sup_{\theta\in \mathbb{R}}\left\|\frac{e^{i\theta}T+(e^{i\theta}T)^{\#_A}}{2}\right\|_A.
	\end{align}
	The author then extended the inequality \eqref{p3100} using $A$-numerical radius of $T$, and the same is produced below:
	\begin{align}\label{ineq1}
	\frac{1}{2}\|T\|_A\leq w_A(T)\leq \|T\|_A.
	\end{align}
	Furthermore,  if $T$ is $A$-selfadjoint, then $w_A(T)=\|T\|_A$.
	In 2019,  Moslehian {\it et al.} \cite{MOS} again continued the study of $A$-numerical radius and established some inequalities for $A$-numerical radius. Further generalizations and refinements of $A$-numerical radius are discussed in \cite{Pintu1, Pintu2, {NSD}}. 
	In 2020, Bhunia {\it et al.}  \cite{PINTU} obtained several $\mathbb{A}$-numerical radius inequalities. For more results on $\mathbb{A}$-numerical radius inequalities we refer the reader to visit \cite{feki04,Nirmal2,XYZ,KS}.

	In 2020, the concept of the $A$-spectral radius of $A$-bounded operators was introduced by Feki in \cite{feki01} as follows:
		\begin{equation}\label{newrad}
		r_A(T):=\displaystyle\inf_{n\geq 1}\|T^n\|_A^{\frac{1}{n}}=\displaystyle\lim_{n\to\infty}\|T^n\|_A^{\frac{1}{n}}.
		\end{equation}
		Here we want to mention that the proof of the second equality in \eqref{newrad} can also be found in \cite[Theorem 1]{feki01}. Like the classical spectral radius of Hilbert space operators, it was shown in \cite{feki01} that $r_A(\cdot)$ satisfies the commutativity property, i.e.
		\begin{equation}\label{commut}
		r_A(TS)=r_A(ST),
		\end{equation}
		for all $T,S\in {\mathcal{L}}_{A^{1/2}}(\mathcal{H})$. For the sequel, if $A=I$, then $\|T\|$, $r(T)$ and $\omega(T)$ denote respectively the classical operator norm, the spectral radius and the numerical radius of an operator $T$.

	The  objective of this paper is to present a few new $\mathbb{A}$-numerical radius inequalities for  $2\times 2$ operator matrices. 
	In this aspect, the rest of the paper is broken down as follows. In section 2, we collect a few results about $\mathbb{A}$-numerical radius inequalities which are required to state and prove the results in the subsequent section. Section 3 contains our main results, and is of two parts. Motivated by the work of Hirzallah et al. \cite{TY}, the first part presents several $\mathbb{A}$-numerical radius inequalities of $2\times 2$ operator matrices while the next part focuses on some $A$-numerical radius inequalities.

 \section{Preliminaries}
We need the following lemmas to prove our results.
\begin{lemma}\label{lm0}\textnormal{[Theorem 7 and corollary 2, \cite{feki01}]}
If $T\in{\mathcal{L}}_{A^{1/2}}(\mathcal{H})$.Then 
\begin{align}\label{Aeq100}
w_{A}(T)\leq \frac{1}{2}(\|T\|_A+\|T^2\|_A^{1/2}).
\end{align}
Further, if  $AT^2=0$, then 
\begin{equation}\label{Aeq101}
w_A(T)= \frac{\|T\|_A}{2}.
\end{equation}
\end{lemma}
\begin{lemma}\label{ll2020}\textnormal{[Corollary 3, \cite{feki01}]}
	Let $T\in{\mathcal{L}}(\mathcal{H})$ is an $A$-self-adjoint operator. Then,
	\begin{equation*}
	\|T\|_{A}=w_A(T)=r_A(T).
	\end{equation*}
\end{lemma}
\begin{lemma}\label{lm5}\textnormal{[Lemma 6, \cite{bfeki}]}
	Let $T=\begin{pmatrix}
	T_1&T_2 \\
	T_3&T_4
	\end{pmatrix}$ be such that $T_1,T_2,T_3,T_4\in {\mathcal{L}}_{A^{1/2}}(\mathcal{H})$. Then, $T\in {\mathcal{L}}_{\mathbb{A}^{1/2}}(\mathcal{H}\oplus \mathcal{H})$ and
	$$ r_\mathbb{A}\left(T\right)\leq r\left[\begin{pmatrix}
	\|T_1\|_A & \|T_2\|_A \\
	\|T_3\|_A & \|T_4\|_A
	\end{pmatrix}\right].$$
\end{lemma}

The following lemma is already proved by Bhunia et al. \cite{PINTU} for the case strictly positive operator $A$. Very recentely the same result proved by Rout et al. \cite{Nirmal2} without the condition $A>0$  is stated next for our purpose.

 \begin{lemma}\label{lem0001}\textnormal{[Lemma 2.4, \cite{Nirmal2}]} 
Let $T_1, T_2\in {\mathcal{L}}_A(\mathcal{H}).$ Then 
\begin{enumerate}
        \item [\textnormal{(i)}]
     $w_{\mathbb{A}}\left(\begin{bmatrix}
    T_1 & O\\
    O & T_2
    \end{bmatrix}\right)= \max\{w_A(T_1), w_A(T_2)\}.$\\
    \item [\textnormal{(ii)}] $w_{\mathbb{A}}\left(\begin{bmatrix}
    O & T_1\\
     T_2 & O
    \end{bmatrix}\right)=w_{\mathbb{A}}\left(\begin{bmatrix}
    O & T_2\\
     T_1 & O
    \end{bmatrix}\right).$\\
    \item [\textnormal{(iii)}] $w_{\mathbb{A}}\left(\begin{bmatrix}
    O & T_1\\
     e^{i\theta}T_2 & O
    \end{bmatrix}\right)=w_{\mathbb{A}}\left(\begin{bmatrix}
    O & T_1\\
     T_2 & O
    \end{bmatrix}\right)$ for~any~$\theta\in\mathbb{R}$.\\
    \item [\textnormal{(iv)}]  $w_{\mathbb{A}}\left(\begin{bmatrix}
    T_1 & T_2\\
     T_2 & T_1
    \end{bmatrix}\right)=\max\{w_A(T_1+T_2),w_A(T_1-T_2)\}.$
     In particular, $w_{\mathbb{A}}\left(\begin{bmatrix}
    O & T_2\\
     T_2 & O
    \end{bmatrix}\right)=w_A(T_2).$
\end{enumerate}
\end{lemma}

 The following Lemma is proved by Rout et al. \cite{Nirmal2}.
\begin{lemma}\label{l001}\textnormal{[Lemma 2.2, \cite{Nirmal2}]}
	Let  $T_1, T_2, T_3, T_4\in {\mathcal{L}}_A(\mathcal{H}).$  Then
	\begin{enumerate}
		\item [\textnormal{(i)}] $w_\mathbb{A}\left(\begin{bmatrix}
		T_1 & O\\
		O & T_4
		\end{bmatrix}\right)\leq w_\mathbb{A}\left(\begin{bmatrix}
		T_1 & T_2\\
		T_3 & T_4
		\end{bmatrix}\right).$
		\item [\textnormal{(ii)}] $w_\mathbb{A}\left(\begin{bmatrix}
		O & T_2\\
		T_3 & O
		\end{bmatrix}\right)\leq w_\mathbb{A}\left(\begin{bmatrix}
		T_1 & T_2\\
		T_3 & T_4
		\end{bmatrix}\right).$
	\end{enumerate}
\end{lemma}

\begin{lemma}\label{lemma1}\textnormal{[Lemma 2.4 and Lemma 3.1, \cite{feki04,bfeki}]}
	Let $T_1, T_4\in{\mathcal{L}}_{A^{1/2}}(\mathcal{H})$. Then, the following assertions hold
	\begin{itemize}
		\item[(i)] ${\left\|\begin{pmatrix}
			T_1 & 0 \\
			0 & T_4
			\end{pmatrix}\right\|}_{\mathbb{A}} = {\left\|\begin{pmatrix}
			0 & T_1 \\
			T_4 & 0
			\end{pmatrix}\right\|}_{\mathbb{A}} = \max\big\{{\|T_1\|}_{A}, {\|T_4\|}_{A}\big\}$.
		\item[(ii)] If $T_1, T_2, T_3, T_4\in{\mathcal{L}}_{A}(\mathcal{H})$, then ${\begin{pmatrix}
			T_1 & T_2 \\
			T_3 & T_4
			\end{pmatrix}}^{\#_A} = \begin{pmatrix}
		T_1^{\#_A} & T_3^{\#_A} \\
		T_2^{\#_A} & T_4^{\#_A}
		\end{pmatrix}$.
	\end{itemize}
\end{lemma}
In order to prove our main result the following identity is essential for our purpose. If $T\in{\mathcal{L}}_{A^{1/2}}(\mathcal{H})$ and
$ \begin{bmatrix}
	T & T\\
	-T & -T
\end{bmatrix}^2=\begin{bmatrix}
0 & 0\\
0 & 0
\end{bmatrix}$, so by \eqref{Aeq101} 

\begin{align}\label{Aeq102}
w_{\mathbb{A}}\left(\begin{bmatrix}
	T & T\\
-T & -T
\end{bmatrix}\right)=\frac{1}{2}\left\|\begin{bmatrix}
T & T\\
-T & -T
\end{bmatrix}\right\|_A=\|T\|_A.
\end{align}

\section{Results}
	
We will split our results into two subsections. The first part deals with ${\mathbb{A}}$-numerical radius of  $2\times 2$ operator matrices.  The second part concerns some upper bound for $A$ numerical radius inequalities. 

\subsection{Certain ${\mathbb{A}}$-numerical radius inequalities of operator matrices}
Here, we establish our main results dealing with different upper and lower bounds for ${\mathbb{A}}$-numerical radius of  $2\times 2$ block operator matrices. The very first result is stated next.
\begin{theorem}\label{Athm100}
Let	$T_2,T_3\in {\mathcal{L}}_A(\mathcal{H})$.  Then
	\begin{align*}
	w_{\mathbb{A}}\left(\begin{bmatrix}
	0 & T_2\\
	T_3 & 0
	\end{bmatrix}\right)\leq \min\left\{w_A(T_2), w_A(T_3)\right\}+\min\left\{\frac{\|T_2+T_3\|_A}{2}, \frac{\|T_2-T_3\|_A}{2}\right\}.
	\end{align*}
\end{theorem}

\begin{proof}
		Let $U=\frac{1}{\sqrt{2}}\begin{bmatrix}
	I & -I\\
	I & I
	\end{bmatrix}$.
	To show that $U$ is $\mathbb{A}$-unitary, we need to prove that $\|x\|_\mathbb{A}=\|Ux\|_\mathbb{A}=\|U^{\#_\mathbb{A}}x\|_\mathbb{A}.$ So,
	\begin{align*}
	U^{\#_\mathbb{A}}&=  \mathbb{A}^\dagger U^{*} \mathbb{A}\\
	&=\frac{1}{\sqrt{2}}\begin{bmatrix}
	A^{\dagger} & O\\
	O & A^{\dagger}
	\end{bmatrix}\begin{bmatrix}
	I & I\\
	-I & I
	\end{bmatrix}\begin{bmatrix}
	A & O\\
	O & A
	\end{bmatrix}\\
	&=\frac{1}{\sqrt{2}}\begin{bmatrix}
	A^{\dagger}A & A^{\dagger}A\\
	-A^{\dagger}A & A^{\dagger}A
	\end{bmatrix}\\
	&=\frac{1}{\sqrt{2}}\begin{bmatrix}
	P_{\overline{\mathcal{R}(A)}} & P_{\overline{\mathcal{R}(A)}}\\
	-P_{\overline{\mathcal{R}(A)}} & P_{\overline{\mathcal{R}(A)}}
	\end{bmatrix} ~~~\because ~~ N(A)^{\perp}=\overline{\mathcal{R}(A^*)} ~~\&~~ \mathcal{R}(A^*)=\mathcal{R}(A).
	\end{align*}
	This in turn implies
	$UU^{\#_\mathbb{A}}=\begin{bmatrix}
	P_{\overline{\mathcal{R}(A)}} & O\\
	O & P_{\overline{\mathcal{R}(A)}}
	\end{bmatrix}= U^{\#_\mathbb{A}}U$.
	Now, for $x=(x_1,x_2)\in \mathcal{H}\bigoplus \mathcal{H}$, we have
	
	\begin{align*}
	\|Ux\|_\mathbb{A}^2=\langle Ux, Ux\rangle_\mathbb{A}=\langle U^{\#_\mathbb{A}}Ux,x\rangle_\mathbb{A} &=\left\langle \begin{bmatrix}
	P_{\overline{\mathcal{R}(A)}} & O\\
	O & P_{\overline{\mathcal{R}(A)}}
	\end{bmatrix}\begin{bmatrix}
	x_1\\
	x_2
	\end{bmatrix},\begin{bmatrix}
	x_1\\
	x_2
	\end{bmatrix}\right\rangle_\mathbb{A}\\
	&=\left\langle\begin{bmatrix}
	AP_{\overline{\mathcal{R}(A)}} & O\\
	O & AP_{\overline{\mathcal{R}(A)}}
	\end{bmatrix}\begin{bmatrix}
	x_1\\
	x_2
	\end{bmatrix},\begin{bmatrix}
	x_1\\
	x_2
	\end{bmatrix}\right\rangle\\
	& = \left\langle\begin{bmatrix}
	AA^\dagger A & O\\
	O & AA^\dagger A
	\end{bmatrix}\begin{bmatrix}
	x_1\\
	x_2
	\end{bmatrix},\begin{bmatrix}
	x_1\\
	x_2
	\end{bmatrix}\right\rangle\\
	&=\left\langle\begin{bmatrix}
	A & O\\
	O & A
	\end{bmatrix}\begin{bmatrix}
	x_1\\
	x_2
	\end{bmatrix},\begin{bmatrix}
	x_1\\
	x_2
	\end{bmatrix}\right\rangle\\
	&=\|x\|_{\mathbb{A}}^2.
	\end{align*}
	So, $\|Ux\|_\mathbb{A}=\|x\|_\mathbb{A}.$ Similarly, it can be proved that $\|U^{\#_\mathbb{A}}x\|_\mathbb{A}=\|x\|_\mathbb{A}.$ Thus, $U$ is an $\mathbb{A}$-unitary operator.  
	
	 Using the identity $w_{\mathbb{A}}(T)=w_{\mathbb{A}}(U^{\#_{\mathbb{A}}}TU)$, we have
	\begin{align}
	w_{\mathbb{A}}\left(\begin{bmatrix}
	0 & T_2\\
	T_3 & 0
	\end{bmatrix}\right)=w_{\mathbb{A}}\left(\begin{bmatrix}
	0&T_2\\T_3&0
	\end{bmatrix}^{\#_{\mathbb{A}}}\right)=& w_{\mathbb{A}}\left(U^{\#_{\mathbb{A}}}\begin{bmatrix}
	0 & T_2\\
	T_3 & 0
	\end{bmatrix}^{\#_{\mathbb{A}}}U\right)\nonumber\\
	&=\frac{1}{2}w_{\mathbb{A}}\left(\begin{bmatrix}
	I&-I\\I&I
	\end{bmatrix}^{\#_{\mathbb{A}}}\begin{bmatrix}
	0&T_3^{\#_A}\\T_2^{\#_A}&0
	\end{bmatrix}\begin{bmatrix}
	I&-I\\I&I
	\end{bmatrix}\right)\nonumber\\
	&=\frac{1}{2}w_{\mathbb{A}}\left(\begin{bmatrix}
	P_{\overline{\mathcal{R}(A)}}&P_{\overline{\mathcal{R}(A)}}\\-P_{\overline{\mathcal{R}(A)}}&P_{\overline{\mathcal{R}(A)}}
	\end{bmatrix}\begin{bmatrix}
	0&T_3^{\#_A}\\T_2^{\#_A}&0
	\end{bmatrix}\begin{bmatrix}
	I&-I\\I&I
	\end{bmatrix}\right) \nonumber\\
	&=\frac{1}{2}w_{\mathbb{A}}\left(\begin{bmatrix}
	P_{\overline{\mathcal{R}(A)}}&P_{\overline{\mathcal{R}(A)}}\\-P_{\overline{\mathcal{R}(A)}}&P_{\overline{\mathcal{R}(A)}}
	\end{bmatrix}\begin{bmatrix}
	T_3^{\#_A}&T_3^{\#_A}\\T_2^{\#_A}&-T_2^{\#_A}
	\end{bmatrix}\right)\nonumber\\
	&=\frac{1}{2}w_{\mathbb{A}}\left(\begin{bmatrix}
	T_3^{\#_A}+T_2^{\#_A}&T_3^{\#_A}-T_2^{\#_A}\\-T_3^{\#_A}+T_2^{\#_A}&-T_3^{\#_A}-T_2^{\#_A}
	\end{bmatrix}\right)\nonumber\\
		&=\frac{1}{2}w_{\mathbb{A}}\left(\begin{bmatrix}
	T_2+T_3 & T_2-T_3\\
	-(T_2-T_3) & -(T_2+T_3)
	\end{bmatrix}^{\#_{\mathbb{A}}}\right)\nonumber\\
		&=\frac{1}{2}w_{\mathbb{A}}\left(\begin{bmatrix}
	T_2+T_3 & T_2-T_3\\
	-(T_2-T_3) & -(T_2+T_3)
	\end{bmatrix}\right)~~ (as~ w_A(T)=w_A(T^{\#_{\mathbb{A}}}))\nonumber\\
	&=\frac{1}{2}w_{\mathbb{A}}\left(\begin{bmatrix}
	T_2+T_3 & T_2+T_3\\
	-(T_2+T_3) & -(T_2+T_3)
	\end{bmatrix}+\begin{bmatrix}
	0 & -2T_3\\
	2T_3 & 0
	\end{bmatrix}\right)\nonumber\\
	&\leq \frac{1}{2}\left\{w_{\mathbb{A}}\left(\begin{bmatrix}
T_2+T_3 & T_2+T_3\\
-(T_2+T_3) & -(T_2+T_3)
	\end{bmatrix}\right)+w_{\mathbb{A}}\left(\begin{bmatrix}
	0 & -2T_3\\
2T_3 & 0
	\end{bmatrix}\right)\right\}\nonumber
	\end{align}
	Now, using  identity \eqref{Aeq102} and Lemma \ref{lem0001}, we have
	\begin{align}\label{Aeq103}
	w_{\mathbb{A}}\left(\begin{bmatrix}
0 & T_2\\
T_3 & 0
\end{bmatrix}\right)\leq \frac{\|T_2+T_3\|_A}{2}+w_A(T_3).
	\end{align}
	Replacing $T_3$ by $-T_3$ in the inequality \eqref{Aeq103} and using Lemma \ref{lem0001}, we get 
	\begin{align}\label{Aeq104}
		w_{\mathbb{A}}\left(\begin{bmatrix}
	0 & T_2\\
	T_3 & 0
	\end{bmatrix}\right)\leq\frac{\|T_2-T_3\|_A}{2}+w_A(T_3).
	\end{align}
	From the inequalities \eqref{Aeq103} and \eqref{Aeq104}, we have 
		\begin{align}\label{Aeq105}
	w_{\mathbb{A}}\left(\begin{bmatrix}
	0 & T_2\\
	T_3 & 0
	\end{bmatrix}\right)\leq  w_A(T_3)+\min\left\{\frac{\|T_2+T_3\|_A}{2}, \frac{\|T_2-T_3\|_A}{2}\right\}.
	\end{align}
	Again, in the inequality \eqref{Aeq105}, interchanging $T_2$ and $T_3$ and using Lemma \ref{lem0001}(ii), we get
		\begin{align}\label{Aeq106}
	w_{\mathbb{A}}\left(\begin{bmatrix}
	0 & T_2\\
	T_3 & 0
	\end{bmatrix}\right)\leq  w_A(T_2)+\min\left\{\frac{\|T_2+T_3\|_A}{2}, \frac{\|T_2-T_3\|_A}{2}\right\}.
	\end{align}
	From the inequalities \eqref{Aeq105} and \eqref{Aeq106}, we get
	\begin{align*}
	w_{\mathbb{A}}\left(\begin{bmatrix}
	0 & T_2\\
	T_3 & 0
	\end{bmatrix}\right)\leq \min\left\{w_A(T_2), w_A(T_3)\right\}+\min\left\{\frac{\|T_2+T_3\|_A}{2}, \frac{\|T_2-T_3\|_A}{2}\right\}.
	\end{align*}
	This completes the proof.
\end{proof}

\begin{theorem}\label{Athm101}
	Let	$T_2,T_3\in {\mathcal{L}}_A(\mathcal{H})$.  Then
	\begin{align*}
	w_{\mathbb{A}}\left(\begin{bmatrix}
	0 & T_2\\
	T_3 & 0
	\end{bmatrix}\right)\geq \max\left\{w_A(T_2), w_A(T_3)\right\}-\min\left\{\frac{\|T_2+T_3\|_A}{2}, \frac{\|T_2-T_3\|_A}{2}\right\}.
	\end{align*}
	and 
	 \begin{align*}
	w_{\mathbb{A}}\left(\begin{bmatrix}
0 & T_2\\
T_3 & 0
\end{bmatrix}\right)\geq	\max\left\{\frac{\| T_2+T_3\|_A}{2}, \frac{\| T_2-T_3\|_A}{2}\right\}-\min\{w_A(T_2), w_A(T_3)\}.
	\end{align*}
\end{theorem}

\begin{proof}
	Let $U=\frac{1}{\sqrt{2}}\begin{bmatrix}
I & -I\\
I & I
\end{bmatrix}$. It can be shown that $U$ is ${\mathbb{A}}$-unitary. Then 	
	
	\begin{align}\label{A1}
	\frac{1}{2}\begin{bmatrix}
	T_2+T_3 & T_2+T_3\\
	-(T_2+T_3) & -(T_2+T_3)
	\end{bmatrix}^{\#_{\mathbb{A}}}= U^{\#_{\mathbb{A}}}\begin{bmatrix}
	0 & T_2\\
	T_3 & 0
	\end{bmatrix}^{\#_{\mathbb{A}}}U-	\begin{bmatrix}
	0 & -T_3\\
	T_3 & 0
	\end{bmatrix}^{\#_{\mathbb{A}}}
	.
	\end{align} 
	So, 
	\begin{align}\label{Aeq107}
	\begin{bmatrix}
0 & -T_3\\
T_3 & 0
\end{bmatrix}^{\#_{\mathbb{A}}}=	 U^{\#_{\mathbb{A}}}\begin{bmatrix}
	0 & T_2\\
	T_3 & 0
	\end{bmatrix}^{\#_{\mathbb{A}}}U-\frac{1}{2}\begin{bmatrix}
	T_2+T_3 & T_2+T_3\\
	-(T_2+T_3) & -(T_2+T_3)
	\end{bmatrix}^{\#_{\mathbb{A}}}.
	\end{align}
	This implies 
	\begin{align*}
		w_{\mathbb{A}}\left(\begin{bmatrix}
	0 & -T_3\\
	T_3 & 0
	\end{bmatrix}^{\#_{\mathbb{A}}}\right)\leq	 	w_{\mathbb{A}}\left(U^{\#_{\mathbb{A}}}\begin{bmatrix}
	0 & T_2\\
	T_3 & 0
	\end{bmatrix}^{\#_{\mathbb{A}}}U\right)+\frac{1}{2}w_{\mathbb{A}}\left(\begin{bmatrix}
	T_2+T_3 & T_2+T_3\\
	-(T_2+T_3) & -(T_2+T_3)
	\end{bmatrix}^{\#_{\mathbb{A}}}\right).
	\end{align*}
	Which in turn implies that 
		\begin{align*}
	w_{\mathbb{A}}\left(\begin{bmatrix}
	0 & -T_3\\
	T_3 & 0
	\end{bmatrix}\right)&\leq	 	w_{\mathbb{A}}\left(\begin{bmatrix}
	0 & T_2\\
	T_3 & 0
	\end{bmatrix}^{\#_{\mathbb{A}}}\right)+\frac{1}{2}w_{\mathbb{A}}\left(\begin{bmatrix}
	T_2+T_3 & T_2+T_3\\
	-(T_2+T_3) & -(T_2+T_3)
	\end{bmatrix}\right)\\
	& =	w_{\mathbb{A}}\left(\begin{bmatrix}
	0 & T_2\\
	T_3 & 0
	\end{bmatrix}\right)+\frac{1}{2}w_{\mathbb{A}}\left(\begin{bmatrix}
	T_2+T_3 & T_2+T_3\\
	-(T_2+T_3) & -(T_2+T_3)
	\end{bmatrix}\right).
	\end{align*}
	Thus, using inequality \eqref{Aeq102} and Lemma \ref{lem0001}
		\begin{align}\label{Aeq108}
	w_A(T_3)\leq	 	w_{\mathbb{A}}\left(\begin{bmatrix}
	0 & T_2\\
	T_3 & 0
	\end{bmatrix}\right)+\frac{\|T_2+T_3\|_A}{2} .
	\end{align}
	Replacing $T_3$ by $-T_3$ in the inequality \eqref{Aeq108} we have 
		\begin{align}\label{Aeq109}
	w_A(T_3)\leq	 	w_{\mathbb{A}}\left(\begin{bmatrix}
	0 & T_2\\
	T_3 & 0
	\end{bmatrix}\right)+\frac{\|T_2-T_3\|_A}{2} .
	\end{align}
	Now from inequality \eqref{Aeq108} and \eqref{Aeq109} that 
		\begin{align}\label{Aeq110}
	w_A(T_3)\leq	 	w_{\mathbb{A}}\left(\begin{bmatrix}
	0 & T_2\\
	T_3 & 0
	\end{bmatrix}\right)+\min\left\{\frac{\|T_2+T_3\|_A}{2}, \frac{\|T_2-T_3\|_A}{2}\right\} .
	\end{align}
	Interchanging $T_2$ and $T_3$ in the ininequality \eqref{Aeq110}, we get
		\begin{align}\label{Aeq111}
	w_A(T_2)   \leq	 w_{\mathbb{A}}\left(\begin{bmatrix}
         	0 & T_2\\
	       T_3 & 0
	\end{bmatrix}\right)+\min\left\{\frac{\|T_2+T_3\|_A}{2}, \frac{\|T_2-T_3\|_A}{2}\right\} .
	\end{align}
	From inequalities \eqref{Aeq110} and \eqref{Aeq111}, we have 
		\begin{align}
	\max\{w_A(T_2), w_A(T_3)\}   \leq	 w_{\mathbb{A}}\left(\begin{bmatrix}
	0 & T_2\\
	T_3 & 0
	\end{bmatrix}\right)+\min\left\{\frac{\|T_2+T_3\|_A}{2}, \frac{\|T_2-T_3\|_A}{2}\right\} .
	\end{align}
	Which proves the first inequality.\\
	Again, by identity \eqref{A1} and inequality \eqref{Aeq102} that 
		\begin{align*}
\frac{1}{2}\| T_2+T_3\|_A=	& \frac{1}{2}w_{\mathbb{A}}\left(\begin{bmatrix}
	 T_2+T_3 & T_2+T_3\\
	 -(T_2+T_3) & -(T_2+T_3)
	 \end{bmatrix}\right)\\
	 &= \frac{1}{2}w_{\mathbb{A}}\left(\begin{bmatrix}
	 T_2+T_3 & T_2+T_3\\
	 -(T_2+T_3) & -(T_2+T_3)
	 \end{bmatrix}^{\#_{\mathbb{A}}}\right)\\
	 &\leq	w_{\mathbb{A}}\left(U^{\#_{\mathbb{A}}}\begin{bmatrix}
	0 & T_2\\
	T_3 & 0
	\end{bmatrix}^{\#_{\mathbb{A}}}U\right)+	w_{\mathbb{A}}\left(\begin{bmatrix}
	0 & -T_3\\
	T_3 & 0
	\end{bmatrix}^{\#_{\mathbb{A}}}\right)\\
	&=	w_{\mathbb{A}}\left(\begin{bmatrix}
	0 & T_2\\
	T_3 & 0
	\end{bmatrix}^{\#_{\mathbb{A}}}\right)+	w_{\mathbb{A}}\left(\begin{bmatrix}
	0 & -T_3\\
	T_3 & 0
	\end{bmatrix}\right)\\
	&=w_{\mathbb{A}}\left(\begin{bmatrix}
	0 & T_2\\
	T_3 & 0
	\end{bmatrix}\right)+w_A(T_3) ~~\mbox{by Lemma \ref{lem0001}} .
	\end{align*}
	Thus, 
		\begin{align}\label{Aeq112}
	\frac{1}{2}\| T_2+T_3\|_A\leq	w_{\mathbb{A}}\left(\begin{bmatrix}
	0 & T_2\\
	T_3 & 0
	\end{bmatrix}\right)+w_A(T_3).
	\end{align}
	Replacing $T_3$ by $-T_3$ in the inequality \eqref{Aeq112} and using Lemma \ref{lem0001}, we get
		\begin{align}\label{Aeq113}
	\frac{1}{2}\| T_2-T_3\|_A\leq	w_{\mathbb{A}}\left(\begin{bmatrix}
	0 & T_2\\
	T_3 & 0
	\end{bmatrix}\right)+w_A(T_3).
	\end{align}
	It follows from inequalities \eqref{Aeq112} and \eqref{Aeq113} that 
	\begin{align}\label{Aeq114}
	\max\left\{\frac{\| T_2+T_3\|_A}{2}, \frac{\| T_2-T_3\|_A}{2}\right\}\leq	w_{\mathbb{A}}\left(\begin{bmatrix}
	0 & T_2\\
	T_3 & 0
	\end{bmatrix}\right)+w_A(T_3).
	\end{align}
	Interchanging $T_2$ and $T_3$ in the inequality \eqref{Aeq114} and using Lemma \ref{lem0001}, we get
		\begin{align}\label{Aeq115}
	\max\left\{\frac{\| T_2+T_3\|_A}{2}, \frac{\| T_2-T_3\|_A}{2}\right\}\leq	w_{\mathbb{A}}\left(\begin{bmatrix}
	0 & T_2\\
	T_3 & 0
	\end{bmatrix}\right)+w_A(T_2).
	\end{align}
	 Now combining \eqref{Aeq114} and \eqref{Aeq115}, we have
	 \begin{align}\label{Aeq116}
	 \max\left\{\frac{\| T_2+T_3\|_A}{2}, \frac{\| T_2-T_3\|_A}{2}\right\}-\min\{w_A(T_2), w_A(T_3)\}\leq	w_{\mathbb{A}}\left(\begin{bmatrix}
	 0 & T_2\\
	 T_3 & 0
	 \end{bmatrix}\right).
	 \end{align}
	 This completes the proof.
\end{proof}

\begin{theorem}\label{Athm102}
	Let	$T_2,T_3\in {\mathcal{L}}_A(\mathcal{H})$.  Then
	\begin{align*}
	w_{\mathbb{A}}^2\left(\begin{bmatrix}
	0 & T_2\\
	T_3 & 0
	\end{bmatrix}\right)\geq \frac{1}{2}\bigg\{w_A(T_2 T_3+T_3 T_2), w_A(T_2 T_3-T_3 T_2)\bigg\}.
	\end{align*}
\end{theorem}

\begin{proof}
	Let us consider $A$-unitary operator $U=\begin{bmatrix}
	0 & I\\
	I & 0
	\end{bmatrix}$; $U^{\#_{\mathbb{A}}}=\begin{bmatrix}
	0&P_{\overline{\mathcal{R}(A)}}\\P_{\overline{\mathcal{R}(A)}}&0
	\end{bmatrix}; T=\begin{bmatrix}
	0 & T_2\\
	T_3 & 0
	\end{bmatrix}.$
	Now,
	\begin{align*}
	(T^{\#_{\mathbb{A}}})^2+(U^{\#_{\mathbb{A}}}T^{\#_{\mathbb{A}}}U)^2&=\begin{bmatrix}
	0 & T_3^{\#_A}\\
	T_2^{\#_A} & 0
	\end{bmatrix}^2+\left(\begin{bmatrix}
	0&P_{\overline{\mathcal{R}(A)}}\\P_{\overline{\mathcal{R}(A)}}&0
	\end{bmatrix}\begin{bmatrix}
	0 & T_3^{\#_A}\\
	T_2^{\#_A} & 0
	\end{bmatrix}\begin{bmatrix}
	0 & I\\
	I & 0
	\end{bmatrix}\right)^2\\
		&=\begin{bmatrix}
		 T_3^{\#_A} T_2^{\#_A} &0\\
		0 &  T_2^{\#_A} T_3^{\#_A}
		\end{bmatrix}+\left(\begin{bmatrix}
	0 & T_2^{\#_A}\\
	T_3^{\#_A} & 0
	\end{bmatrix}\right)^2\\
	&=\begin{bmatrix}
	T_3^{\#_A} T_2^{\#_A} &0\\
	0 &  T_2^{\#_A} T_3^{\#_A}
	\end{bmatrix}+\begin{bmatrix}
	T_2^{\#_A} T_3^{\#_A} &0\\
	0 &  T_3^{\#_A} T_2^{\#_A}
	\end{bmatrix}\\
	&=\begin{bmatrix}
	T_3^{\#_A} T_2^{\#_A}+T_2^{\#_A} T_3^{\#_A} &0\\
	0 &  T_2^{\#_A} T_3^{\#_A}+T_3^{\#_A} T_2^{\#_A}
	\end{bmatrix}\\
	&=\begin{bmatrix}
	T_2 T_3+T_3 T_2 &0\\
	0 &  T_3 T_2+T_2 T_3
	\end{bmatrix}^{\#_{\mathbb{A}}}.
	\end{align*}
So,
\begin{align*}
w_{\mathbb{A}}\left(\begin{bmatrix}
T_2 T_3+T_3 T_2 &0\\
0 &  T_3 T_2+T_2 T_3
\end{bmatrix}\right)&=w_{\mathbb{A}}\left(\begin{bmatrix}
T_2 T_3+T_3 T_2 &0\\
0 &  T_3 T_2+T_2 T_3
\end{bmatrix}^{\#_{\mathbb{A}}}\right)\\
&=w_{\mathbb{A}}\left(	(T^{\#_{\mathbb{A}}})^2+(U^{\#_{\mathbb{A}}}T^{\#_{\mathbb{A}}}U)^2\right)\\
&\leq w_{\mathbb{A}}\left(	(T^{\#_{\mathbb{A}}})^2\right)+w_{\mathbb{A}}\left((U^{\#_{\mathbb{A}}}T^{\#_{\mathbb{A}}}U)^2\right)\\
&\leq w_{\mathbb{A}}^2\left(	T^{\#_{\mathbb{A}}}\right)+w_{\mathbb{A}}^2\left(U^{\#_{\mathbb{A}}}T^{\#_{\mathbb{A}}}U\right)\\
&= w_{\mathbb{A}}^2\left(	T^{\#_{\mathbb{A}}}\right)+w_{\mathbb{A}}^2\left(T^{\#_{\mathbb{A}}}\right)\\
&= w_{\mathbb{A}}^2\left(	T\right)+w_{\mathbb{A}}^2\left(T\right)\\
&=2 w_{\mathbb{A}}^2\left(	T\right)~~~\left(as ~w_{\mathbb{A}}(	T)=w_{\mathbb{A}}(	T^{\#_{\mathbb{A}}})\right).
\end{align*}
Hence by using Lemma \ref{lem0001} we obtain
\begin{align}\label{Aeq117}
w_A(T_2 T_3+T_3 T_2)\leq 2 w_{\mathbb{A}}^2\left(T\right).
\end{align} 
Using similar argument to $(T^{\#_{\mathbb{A}}})^2-(U^{\#_{\mathbb{A}}}T^{\#_{\mathbb{A}}}U)^2$, we have
\begin{align}\label{Aeq118}
w_A(T_2 T_3-T_3 T_2)\leq 2 w_{\mathbb{A}}^2\left(T\right).
\end{align} 
Combining \eqref{Aeq117} and \eqref{Aeq118} we get
\begin{align*}
w_{\mathbb{A}}^2\left(\begin{bmatrix}
0 & T_2\\
T_3 & 0
\end{bmatrix}\right)\geq \frac{1}{2}\bigg\{w_A(T_2 T_3+T_3 T_2), w_A(T_2 T_3-T_3 T_2)\bigg\}.
\end{align*}
\end{proof}

\begin{cor}\label{Acor100}
	Let	$T_1,T_2, T_3, T_4\in {\mathcal{L}}_A(\mathcal{H})$.  Then
	\begin{align*}
	w_{\mathbb{A}}\left(\begin{bmatrix}
	T_1& T_2\\
T_3 & T_4
	\end{bmatrix}\right)\geq \max\bigg\{w_A(T_1), w_A(T_4), \frac{1}{\sqrt{2}}\left(w_A(T_2T_3+T_3T_2)\right)^\frac{1}{2}, \frac{1}{\sqrt{2}}\left(w_A(T_2T_3-T_3T_2)\right)^\frac{1}{2}\bigg\}.
	\end{align*}
\end{cor}
\begin{proof}
Based on Lemma \ref{l001}, Lemma \ref{lem0001} and Theorem \ref{Athm102} we have 
		\begin{align*}
	w_{\mathbb{A}}\left(\begin{bmatrix}
	T_1& T_2\\
	T_3 & T_4
	\end{bmatrix}\right)&\geq \max\left\{	w_{\mathbb{A}}\left(\begin{bmatrix}
	T_1 & 0\\
	0 &T_4
	\end{bmatrix}\right), 	w_{\mathbb{A}}\left(\begin{bmatrix}
	0 & T_2\\
	T_3 & 0
	\end{bmatrix}\right)\right\}\\
	&\geq \max\bigg\{w_A(T_1), w_A(T_4), \frac{1}{\sqrt{2}}\left(w_A(T_2T_3+T_3T_2)\right)^\frac{1}{2}, \frac{1}{\sqrt{2}}\left(w_A(T_2T_3-T_3T_2)\right)^\frac{1}{2}\bigg\}.
	\end{align*}
\end{proof}

\begin{theorem}\label{Athm104}
	Let	$T_2,T_3\in {\mathcal{L}}_A(\mathcal{H})$. Then for $n\in \mathbb{N}$
	\begin{align}\label{Aeq119}
	w_{\mathbb{A}}\left(\begin{bmatrix}
	0 &T_2\\
	T_3 & 0
	\end{bmatrix}\right)\geq \left[\max\{w_A((T_2T_3)^n), w_A((T_3T_2)^n)\}\right]^{\frac{1}{2n}}.
	\end{align}
\end{theorem}
\begin{proof}
	Let $T=\begin{bmatrix}
	0 &T_2\\
	T_3 & 0
	\end{bmatrix}$. Then for $n\in \mathbb{N}$, $T^{2n}=\begin{bmatrix}
	(T_2T_3)^n &0\\
	 0& (T_3T_2)^n 
	\end{bmatrix}$ and using Lemma \ref{lem0001} we  obtain 
	\begin{align*}
	\max\{w_A((T_2T_3)^n), w_A((T_3T_2)^n)\}&=w_{\mathbb{A}}\left(\begin{bmatrix}
	(T_2T_3)^n &0\\
	0& (T_3T_2)^n 
	\end{bmatrix}\right)\\
	&=w_{\mathbb{A}}(T^{2n})\\
	&\leq w_{\mathbb{A}}^{2n}(T) ~~~\mbox{by inequality \ref{power}}\\
	&=w_{\mathbb{A}}^{2n}\left(\begin{bmatrix}
	0 &T_2\\
	T_3 & 0
	\end{bmatrix}\right).
	\end{align*}
\end{proof}
The following lemma is already proved by Hirzallah et al. \cite{TY} for the case of Hilbert space operators. Using similar techinque we can prove this lemma for the case of semi-Hilbert space. Now we state here the result without proof for our purpose.
\begin{lemma}\label{lem0002}
	Let $T=\begin{bmatrix}
	T_1&T_2\\
	T_2 &T_1
	\end{bmatrix}\in {\mathcal{L}}_A(\mathcal{H}\oplus \mathcal{H})$ and $n\in \mathbb{N}$. Then $T^n=\begin{bmatrix}
	P & Q\\
    Q & P
	\end{bmatrix} $ for some $P, Q\in {\mathcal{L}}_A\mathcal{(H)}$ such that $P+Q=(T_1+T_2)^n$ and $P-Q=(T_1-T_2)^n$.
\end{lemma}

The forthcoming result is analogous to Theorem \ref{Athm104}
\begin{theorem}\label{Athm105}
	Let $T_1, T_2\in {\mathcal{L}}_A(\mathcal{H})$. Then
	\begin{align}\label{Aeq120}
		w_{\mathbb{A}}\left(\begin{bmatrix}
	T_1 &T_2\\
	-T_2& -T_1
	\end{bmatrix}\right)\geq \big[\max \{w_A\left(((T_1-T_2)(T_1+T_2))^n\right), w_A\left(((T_1+T_2)(T_1-T_2))^n\right)\}\big]^{\frac{1}{2n}}
	\end{align}
for $n\in \mathbb{N}$ and 
\begin{align}\label{Aeq121}
w_{\mathbb{A}}\left(\begin{bmatrix}
T_1 &T_2\\
-T_2& -T_1
\end{bmatrix}\right)&\leq \frac{\max\{\|T_1+T_2\|_A, \|T_1-T_2\|_A\}}{2}\nonumber\\
&+\frac{[\max\{\|(T_1+T_2)(T_1-T_2)\|_A, \|(T_1-T_2)(T_1+T_2)\|_A\}]^{\frac{1}{2}}}{2}.
\end{align}
\end{theorem}
\begin{proof}
Let $T=\begin{bmatrix}
T_1 &T_2\\
-T_2& -T_1
\end{bmatrix}$ and $R=T^2=\begin{bmatrix}
T_1^2-T_2^2 &T_1T_2-T_2T_1\\
T_1T_2-T_2T_1& T_1^2-T_2^2
\end{bmatrix}$.	
Using Lemma \ref{lem0002} we have there exist $P, Q\in {\mathcal{L}}_A(\mathcal{H})$ such that 
$R^n=\begin{bmatrix}
P & Q\\
Q & P
\end{bmatrix}$ with $P+Q=((T_1^2-T_2^2)+(T_1T_2-T_2T_1))^n$
and $P-Q=((T_1^2-T_2^2)-(T_1T_2-T_2T_1))^n$.
So, $T^{2n}=\begin{bmatrix}
P & Q\\
Q & P
\end{bmatrix}$ with $P+Q=((T_1-T_2)(T_1+T_2))^n$ and $P-Q=((T_1+T_2)(T_1-T_2))^n$.
By using inequality \eqref{power}, we have 
\begin{align}
w_{\mathbb{A}}^{2n}(T)&\geq w_{\mathbb{A}}(T^{2n})\nonumber\\
&=w_{\mathbb{A}}\left(\begin{bmatrix}
P & Q\\
Q & P
\end{bmatrix}\right)\nonumber\\
&=\max\{w_A(P+Q), w_A(P-Q)\} ~~(\mbox{by Lemma \ref{lem0001}})\nonumber\\
&=\max \{w_A\left(((T_1-T_2)(T_1+T_2))^n\right), w_A\left(((T_1+T_2)(T_1-T_2))^n\right)\}.
\end{align}
This proves the inequality \eqref{Aeq120}.
In order to prove the inequality \eqref{Aeq121}, let  $T=\begin{bmatrix}
T_1 &T_2\\
-T_2& -T_1
\end{bmatrix}$. Then $T^{\#\mathbb{A}}=\begin{bmatrix}
T_1^{\#A} &-T_2^{\#A}\\
T_2^{\#A}& -T_1^{\#A}
\end{bmatrix}$, so $TT^{\#\mathbb{A}}=\begin{bmatrix}
T_1T_1^{\#A}+T_2T_2^{\#A} &-T_1T_2^{\#A}-T_2T_1^{\#A}\\
-T_2T_1^{\#A}-T_1T_2^{\#A}& T_2T_2^{\#A} +T_1T_1^{\#A}
\end{bmatrix}$.
Now it follows from \eqref{ineq0} that 
\begin{align*}
\|T\|_{\mathbb{A}}^2&=\|TT^{\#\mathbb{A}}\|_{\mathbb{A}}\\
&=w_A(TT^{\#\mathbb{A}})\\
&=\max\{w_A(T_1T_1^{\#A}+T_2T_2^{\#A}-T_1T_2^{\#A}-T_2T_1^{\#A}), w_A(T_1T_1^{\#A}+T_2T_2^{\#A}+T_1T_2^{\#A}+T_2T_1^{\#A})\}\\
&\hspace{11cm}(\mbox{by Lemma \ref{lem0001}})\\
&=\max\{w_A((T_1-T_2)(T_1-T_2)^{\#A}), w_A((T_1+T_2)(T_1+T_2)^{\#A})\}\\
&=\max\{\|(T_1-T_2)(T_1-T_2)^{\#A}\|_A, \|(T_1+T_2)(T_1+T_2)^{\#A}\|_A\}
\\
&=\max\{\|T_1-T_2\|_A^2, \|T_1+T_2\|_A^2\}.
\end{align*}
Thus 
\begin{align}\label{Aeq122}
\|T\|_{\mathbb{A}}=\max\{\|T_1-T_2\|_A, \|T_1+T_2\|_A\}.
\end{align}
Similarly we can show that 
\begin{align}\label{Aeq123}
\|T^2\|_{\mathbb{A}}=\max\{\|(T_1-T_2)(T_1+T_2)\|_A, \|(T_1+T_2)(T_1-T_2)\|_A\}.
\end{align}
From inequality \eqref{Aeq100}, combining inequality \eqref{Aeq122} and \eqref{Aeq123}, we obtain
\begin{align*}
w_{\mathbb{A}}(T)&\leq \frac{1}{2}(\|T\|_{\mathbb{A}}+\|T^2\|_{\mathbb{A}}^{1/2})\\
&=\frac{\max\{\|T_1+T_2\|_A, \|T_1-T_2\|_A\}}{2}\nonumber\\
&+\frac{[\max\{\|(T_1+T_2)(T_1-T_2)\|_A, \|(T_1-T_2)(T_1+T_2)\|_A\}]^{\frac{1}{2}}}{2}.
\end{align*}
\end{proof}
\subsection{Some $A$-numerical radius inequalities for operators}
In this subsection we establish some upper bounds for $A$-numerical radius of operators.
In the next result, we derive an upper bound for $A$-numerical radius of product of operators on semi-Hilbertian space.
\begin{theorem}\label{Athm103}
	Let	$T_1,T_2\in {\mathcal{L}}_A(\mathcal{H})$. Then 	
	$$w_A(T_1T_2)\leq \frac{1}{2}\bigg(\|T_2T_1\|_A+\|T_1\|_A\|T_2\|_A\bigg).$$
\end{theorem}

\begin{proof}
	It is not difficult to see that $\Re_A(e^{i\theta}T_1T_2)$ is an $A$-selfadjoint operator. So, by Lemma \ref{ll2020} we have
	$${\left\|\Re_A(e^{i\theta}T_1T_2)\right\|}_A=w_A(\Re_A(e^{i\theta}T_1T_2)).$$
	So,
	\begin{align*}
	{\left\|\Re_A(e^{i\theta}T_1T_2)\right\|}_A&=\frac{1}{2}w_A\left(e^{i\theta}T_1T_2+e^{-i\theta}T_2^{\#_A}T_1^{\#_A}\right)\\
	&=\frac{1}{2}w_{\mathbb{A}}\left(\begin{bmatrix}
	e^{i\theta}T_1T_2+e^{-i\theta}T_2^{\#_A}T_1^{\#_A} & 0\\
	0& 0
	\end{bmatrix}\right)
	\end{align*}
	It can observed that 
	\begin{align*}
	\begin{bmatrix}
	A &0\\
	0 &A
	\end{bmatrix}\begin{bmatrix}
	e^{i\theta}T_1T_2+e^{-i\theta}T_2^{\#_A}T_1^{\#_A} & 0\\
	0& 0
	\end{bmatrix}&=\begin{bmatrix}
	e^{i\theta}AT_1T_2+e^{-i\theta}AT_2^{\#_A}T_1^{\#_A} & 0\\
	0& 0
	\end{bmatrix}\\
	&=\begin{bmatrix}
	e^{i\theta}(T_2^{\#_A}T_1^{\#_A})^*A+e^{-i\theta}(T_1T_2)^*A & 0\\
	0& 0
	\end{bmatrix}
	\\
	&=\begin{bmatrix}
	e^{-i\theta}T_2^{\#_A}T_1^{\#_A}+e^{i\theta}T_1T_2 & 0\\
	0& 0
	\end{bmatrix}^*\begin{bmatrix}
	A &0\\
	0 &A
	\end{bmatrix}
	\end{align*}
	Hence $\begin{bmatrix}
	e^{i\theta}T_1T_2+e^{-i\theta}T_2^{\#_A}T_1^{\#_A} & 0\\
	0& 0
	\end{bmatrix}$ is $\mathbb{A}$-selfadjoint operator.\\
	So by applying Lemma \ref{ll2020} we see that
	\begin{align*}
	{\left\|\Re_A(e^{i\theta}T_1T_2)\right\|}_A
	&=\frac{1}{2}r_{\mathbb{A}}\left(\begin{bmatrix}
	e^{i\theta}T_1T_2+e^{-i\theta}T_2^{\#_A}T_1^{\#_A} & 0\\
	0& 0
	\end{bmatrix}\right)\\
	&=\frac{1}{2}r_{\mathbb{A}}\left(\begin{bmatrix}
	e^{i\theta}T_1& T_2^{\#_A}\\
	0& 0
	\end{bmatrix}\begin{bmatrix}
	T_2 & 0\\
	e^{-i\theta}T_1^{\#_A}& 0
	\end{bmatrix}\right)
	\end{align*}
	So, by using \eqref{commut} we have
	\begin{align*}
	{\left\|\Re_A(e^{i\theta}T_1T_2)\right\|}_A
	&=\frac{1}{2}r_{\mathbb{A}}\left(\begin{bmatrix}
	T_2 & 0\\
	e^{-i\theta}T_1^{\#_A}& 0
	\end{bmatrix}\begin{bmatrix}
	e^{i\theta}T_1& T_2^{\#_A}\\
	0& 0
	\end{bmatrix}\right)\\
	&=\frac{1}{2}r_{\mathbb{A}}\left(\begin{bmatrix}
	e^{i\theta}T_2T_1 & T_2T_2^{\#_A}\\
	T_1^{\#_A}T_1& T_1^{\#_A}T_2^{\#_A}
	\end{bmatrix}\right)
	\\
	&\leq\frac{1}{2}r\left(\begin{bmatrix}
	\|T_2T_1\|_A & \|T_2T_2^{\#_A}\|_A\\
	\|T_1^{\#_A}T_1\|_A& \|T_1^{\#_A}T_2^{\#_A}\|_A
	\end{bmatrix}\right)~~~(\mbox{by Lemma}~\ref{lm5})\\
	&=\frac{1}{2}\bigg(\|T_2T_1\|_A+\|T_1\|_A\|T_2\|_A\bigg).
	\end{align*}
	So by taking supremum over $\theta \in \mathbb{R}$, then using \ref{zm} we get our desired result.
\end{proof}

\vspace{.6cm}
\noindent
{\small {\bf Acknowledgments.}\\
We  thank the {\bf Government of India} for introducing the {\it work from home initiative} during the COVID-19 pandemic.
}
\section{References}
	\bibliographystyle{amsplain}

\end{document}